\documentclass[11pt]{amsart}
\usepackage{enumerate,amssymb,amsmath,latexsym,verbatim,color}
\newcommand\RR{\mathbb{R}}

\newcommand\ZZ{\mathbb{Z}}

\newcommand\CI{{\mathcal C}^{\infty}}

\newcommand\pa{\partial}
\newcommand{\del}{\partial}
\newcommand{\calM}{\mathcal M}
\newcommand{\calO}{\mathcal O}
\newcommand{\calMb}{\overline{\calM}}
\newcommand{\WP}{\mathrm{WP}} 
\newcommand{\calC}{\mathcal C}
\newcommand{\calD}{\mathcal D}
\newcommand{\calU}{\mathcal U}
\newcommand{\calV}{\mathcal V}
\newcommand{\calW}{\mathcal W}

\newcommand{\orb}{\mathrm{orb}}

\newcommand\Ran{\operatorname{Ran}}

\newcommand\ep{\epsilon}

\newcommand\supp{\operatorname{supp}}
\newcommand\loc{\mathrm{loc}}

\newcommand\cM{\mathcal{M}}
\newcommand\cU{\mathcal{U}}

\newtheorem{thm}{Theorem}

\newtheorem{lem}[]{Lemma}
\newtheorem{cor}[]{Corollary}

\newtheorem{rem}[]{Remark}

\newcounter{mnotecount}[section]

\title[Weil-Petersson spectral theory]{Spectral theory \\ for the Weil-Petersson Laplacian \\ on the Riemann moduli space}
\author{Lizhen Ji, Rafe Mazzeo, Werner M\"uller and Andras Vasy}

\date{June 16, 2012}
\begin{document}

\maketitle

\begin{abstract}
We study the spectral geometric properties of the scalar Laplace-Beltrami operator associated to the 
Weil-Petersson metric $g_{\WP}$ on $\calM_\gamma$, the Riemann moduli space of surfaces of genus $\gamma > 1$. 
This space has a singular compactification with respect to $g_{\WP}$, and this metric has crossing cusp-edge singularities 
along a finite collection of simple normal crossing divisors.  We prove first that the scalar Laplacian is
essentially self-adjoint, which then implies that its spectrum is discrete. The second theorem is 
a Weyl asymptotic formula for the counting function for this spectrum. 
\end{abstract}

\section{Introduction}
This paper initiates the analytic study of the natural geometric operators associated to the Weil-Petersson metric 
$g_{\WP}$ on the Riemann moduli space of surfaces of genus $\gamma > 1$, denoted below by $\calM_\gamma$.  We 
consider the Deligne-Mumford compactification of this space, $\calMb_\gamma$, which is a stratified space with many 
special features. The topological and geometric properties of this space have been intensively investigated for many 
years, and $\calM_\gamma$ plays a central role in many parts of mathematics. 

This paper stands as the first step in a development of analytic techniques to study the natural elliptic differential 
operators associated to the Weil-Petersson metric. More generally, our results here apply to the wider class of
metrics with crossing cusp-edge singularities on certain stratified Riemannian pseudomanifolds.  This fits 
into a much broader study of geometric analysis on stratified spaces using the techniques of geometric
microlocal analysis. Some of this work is directed toward studying general classes of such spaces, 
while other parts are focused on specific problems arising on particular singular spaces which arise `in nature', 
such as compactifications of geometric moduli spaces, etc.  These approaches are, of course, closely intertwined.
The perspective of this paper is that $(\calM_\gamma, g_{\WP})$ is inherently interesting, and that the spectral theory 
of its Laplacian will most likely find interesting applications; at the same time, it is an interesting
challenge to develop analytic techniques which can be used to study other singular spaces with
related metric structures.

Our goals here are relatively modest. As stated above, we focus on the scalar Laplacian $\Delta$,
rather than any more complicated operator, associated to the Weil-Petersson metric on $\calM_\gamma$, 
and provide answers to the most basic analytic questions about this operator. 
\begin{thm}
The scalar Laplace operator $\Delta$ on $(\calM_\gamma, g_{\WP})$ is essentially self-adjoint, i.e.\ 
there is a unique self-adjoint extension from the core domain $\calC^\infty_{0,\orb}(\mathring \calM_\gamma)$.
The spectrum of this operator is discrete, and if $N(\lambda)$ denotes the number of eigenvalues of
$\Delta$ which are less than $\lambda$, then 
\[
N(\lambda) = \frac{\omega_n}{(2\pi)^n } \mbox{Vol}_{\WP}(\calM) \lambda^{n/2} + o(\lambda^{n/2}).
\]
as $\lambda \to \infty$.  Here, as usual, $\omega_n$ is the volume of the unit 
ball in $\RR^n$.
\end{thm}
\begin{rem}
There is a subtlety in the statement of this theorem which we point out immediately. The {\it interior} of the 
space $\calM_\gamma$ already has singularities, but these are caused not by any properties of the Weil-Petersson
metric, but rather are orbifold points which arise because the mapping class group does not act freely
on Teichm\"uller space.  While an analysis of the self-adjoint extensions near such points can
be carried out, we instead restrict to an easily defined core domain which fixes the nature of the self-adjoint
extension near these points.  This is defined as follows. If $p$ is a singular point in the interior of $\calM_\gamma$,
then there is a neighbourhood $\calU_p$ around $p$, an open set $\tilde{\calU}_p$ in $\RR^{N}$ ($N = 6\gamma-6$),
and a finite group $\Gamma_p$ which acts on $\tilde{\calU}_p$ such that $\tilde{\calU}_p/\Gamma_p = \calU_p$. 
We then define $\calC^\infty_{0}(\calM_\gamma)$ to consist of all functions $f$ such that the
restriction of $f$ to $\calU_p$ lifts to a $\calC^\infty$ function $\tilde{f}$ on $\tilde{\calU}_p$.  We refer
to \cite{DGGW} for more on this and related other analytic constructions on orbifolds.
Our main result then is that $\Delta$ is essentially self-adjoint on {\it this} core domain. 
The arguments in this paper are essentially all local (or at least localizable), which means that for all
analytic purposes, it suffices to assume that the interior is smooth and that the singular set of
the compactification is of crossing cusp-edge type, as described below, even though the actual
space is locally a finite quotient of this.  
\end{rem}

The emphasis on showing that a given operator is essentially self-adjoint is not so standard in
geometric analysis for the simple reason that this property is automatic for all `reasonable'
elliptic operators on any complete manifold, and when the issue actually arises, e.g.\ on manifolds
with boundary, it is so customary to impose boundary conditions that one rarely thinks of
this as choosing a self-adjoint extension.  On singular spaces, by contrast, the issue
becomes a very real one, and a key preliminary part of the analysis on any such space
is to determine whether the singular set is large enough, in some sense, to create the need
for the imposition of boundary values.  For other classes of singular spaces, e.g.\ those
with cones, edges, etc., this issue is well-understood. It is known that if the singular
set has codimension at least $4$, then there is no need to impose boundary conditions
for the Laplacian on functions.  For the Laplacian on differential forms, however, and for
these same types of `iterated edge metrics', the situation is more complicated, 
and was first considered carefully by Cheeger \cite{Ch}, \cite{Ch2}, see also \cite{ALMP}, \cite{ALMP2}
for some recent work on this.   However, the Weil-Petersson metric is more singular
than these spaces, which leads to the goals of this paper.

The proof of essential self-adjointness for any operator translates to a technical
problem of showing that any element in the maximal domain $\calD_{\max}$ of this operator
is necessarily in the minimal domain $\calD_{\min}$. We review the definitions of these
domains in the beginning of \S 3. This is simply a regularity statement: we wish to
show that any $u \in \calD_{\max}$ enjoys enough regularity and decay near the
singular set to allow us to prove that it can be approximated in graph norm by
elements of $\calC^\infty_{0}(\calM_\gamma)$.  The techniques used to
prove this regularity here are somewhat ad hoc, and in particular do not use any of the 
heavy microlocal machinery which has proved to be very helpful for the study of more 
detailed analytic questions on stratified spaces. The advantage, however, is that this
approach is much more self-contained. 

It is worth recalling the well-known fact that the Laplacian on $\RR^n$ is essentially
self-adjoint on the core-domain $\calC^\infty_0(\RR^n \setminus \{0\})$ if and only
if $n \geq 4$, see \cite{CdV}, and that the $4$-dimensional case has a borderline nature. 
The `radial part' of the Weil-Petersson Laplacian near a divisor is essentially the same 
as the radial part of the Laplacian on $\RR^4$, so in our setting too there are some
borderline effects in the analysis. This motivates our introduction of a slightly broader
class of crossing cusp-edge metrics of any order $k \geq 3$, for any of which
we carry out this analysis. This is intended to clarify the slightly more delicate
argument needed when $k=3$. We also mention the related work on metric
horns (without iterative structure) by Br\"uning, Lesch and
Peyerimhoff; see \cite{Bru, LePe}.

An immediate consequence of the equivalence $\calD_{\max} = \calD_{\min}$ for the scalar
Laplacian is the fact that this domain is compactly contained in $L^2$, which proves immediately
that the spectrum is discrete.  Our final result, concerning the Weyl asymptotics of the counting
function for this spectrum, employs the classical Dirichlet-Neumann bracketing method, hence
does not provide much information about the error term. 

In \S 2 we provide a brief (and sketchy) review of the key properties of the geometric structure of the 
Weil-Petersson metric and of the singular set of $\calMb_\gamma$.  The key fact, that the local lifts of
$g_{\WP}$ have `crossing cusp-edge' singularities of order $3$, leads to the introduction of 
the analogous class of metrics of any order $k \geq 3$. The rest of the paper then considers the scalar
Laplacian for any metric of this more general type.  Essential self-adjointness is studied in \S 3,
and the Weyl estimate is obtained in \S 4. 

\medskip

L.J., R.M. and A.V. gratefully to acknowledge NSF support through the grants DMS-1104696, 1105050 and 
0801226 \& 1068742, respectively. 

\section{The geometry of the Weil-Petersson metric}
We begin with a description of the singular structure of $\calMb_\gamma$ and the structure
of $g_{\WP}$ near the singular strata.   The results about the structure of the Deligne-Mumford
compactification itself are classical at this point, and we refer to \cite{HaMo}, 
\cite{ACG}  for  more detailed
descriptions of all of this and references.  
The form of the Weil-Petersson metric traces back to a paper of 
Masur  \cite{Mas} in the early 1970's, but a far more detailed picture has emerged through
the work of Yamada \cite{Yam} and Wolpert \cite{Wol1}, \cite{Wol2}. We point to two important 
recent survey articles \cite{Yam2} and \cite{Wol3} and the references therein.

The compact space $\calMb_\gamma$ is a complex space which is singular along the union of 
a collection $D_0, \ldots, D_{[\gamma/2]}$ of immersed divisors with simple normal crossing. 
Elements of $\calM_\gamma$ correspond to conformal structures on the underlying compact
surface $\Sigma$ of genus $\gamma$, where conformal structures are identified if they differ
by an arbitrary diffeomorphism of $\Sigma$. Another realization of this space is as the space
of hyperbolic metrics on $\Sigma$ identified by the same space of diffeomorphisms. By contrast,
Teichm\"uller space $\mathcal T_\gamma$ consists of the space of all conformal structures or
hyperbolic metrics identified only by the smaller group of diffeomorphisms of $\Sigma$ isotopic
to the identity.  Thus
\[
\calM_\gamma = \mathcal T_\gamma/\mathrm{Map}(\Sigma),
\]
where the so-called mapping class group $\mathrm{Map}(\Sigma)$ is a discrete group of automorphisms
of $\mathcal T_\gamma$ defined as the quotient of the group of all diffeomorphisms by the subgroup of 
those isotopic to the identity. 

Let $c_1, \ldots, c_N$  be a maximal collection of homotopically nontrivial 
disjoint simple closed curves on $\Sigma$. It is well-known that $N = 3\gamma-3$, and that
$\Sigma \setminus \{c_1, \ldots, c_N\}$ is a union of $2\gamma-2$ pairs
of pants, and moreover, exactly $[\gamma/2] + 1$ of the curves are distinct after identification 
by $\mathrm{Map}(\Sigma)$. 

There is a simple geometric meaning to each of the divisors. Let $D_j$ be the divisor associated
to an equivalence class of curves $[c]$ (i.e.\ curves in this equivalence class are identified, 
up to homotopy, by elements of $\mathrm{Map}\,(\Sigma)$). A sequence of points $p_i \in \calM_\gamma$ 
converges to $D_j$ if the geodesics freely homotopic to $c_j$ for the corresponding sequence of 
hyperbolic metrics have lengths $\ell(c_j) \to 0$. A crossing $D_{j_1} \cap \ldots \cap D_{j_\ell}$ 
corresponds to the independent length degeneration of some collection of equivalence classes of curves 
$c_{j_1}, \ldots, c_{j_\ell}$. In the following, we shall often denote such an $\ell$-fold 
intersection by $D_J$ where $J = \{j_1, \ldots, j_\ell\}$.  Each divisor $D_j$ can be identified with the 
Riemann moduli space for the (possibly disconnected) noded surface $\Sigma'$ obtained by 
pinching the curve $c_j$, or equivalently, by cutting $\Sigma$ along $c_j$ and 
identifying each of the boundaries, which are copies of $c_j$, to points. 

There are a number of natural and interesting metrics on Teichm\"uller space which are invariant
with respect to $\mathrm{Map}(\Sigma)$ and which thus descend to metrics on $\calM_\gamma$. 
One of the most fundamental is the Weil-Petersson metric, $g_{\WP}$, which is the one studied here.
It is incomplete on $\calM_\gamma$ and induces the corresponding Weil-Petersson metric on each 
of the divisors. It is simply the canonical $L^2$ inner product on tangent vectors:  if
$h$ is a hyperbolic metric on $\Sigma$ representing a point of $\calM_\gamma$, then the
tangent space $T_h \calM_\gamma$ is identified with the space of transverse-traceless symmetric
two-tensors $\kappa$ on $\Sigma$, i.e.\ $\mathrm{tr}^h \kappa = 0$ and $\delta^h \kappa = 0$. 
If $\kappa_1$ and $\kappa_2$ are two such tangent vectors, then
\[
\langle \kappa_1, \kappa_2 \rangle_{g_{\WP}} = \int_\Sigma \langle \kappa_1, \kappa_2 \rangle_h \, dA_h = 
\int_\Sigma  (\kappa_1)_{ij} (\kappa_2)_{k\ell} h^{ik} h^{j\ell}\, dA_h. 
\]

It is known that $\calM_{\gamma}$ is a complex orbifold and that $g_{\WP}$ is a K\"ahler metric 
with many interesting properties.  Our main concern is its fine asymptotic structure near the singular divisors, 
which are due to Yamada \cite{Yam} and Wolpert \cite{Wol4}, with closely related results by Liu, Sun and Yau \cite{LSY}. 
Let $p$ be a point in some $D_J$ and choose a local holomorphic coordinate chart $(z_1, \ldots, z_{3g-3})$ with 
$D_J = \{z_1 = \ldots = z_\ell = 0\}$. Setting $z_j = \rho_j e^{i \theta_j}$, $j \leq \ell$, then
\begin{equation}
g_{\mathrm{WP}} = \pi^3 \sum_{j=1}^\ell ( 4 d\rho_j^2 + \rho_j^6 d \theta_j^2)(1 + |\rho|^3) +  g_{D_J}  + 
\calO(|\rho|^3) 
\label{gwp1}
\end{equation}
where $g_{D_J}$ is the Weil-Peterson metric on $D_J$. The expression 
$\calO(|\rho|^3)$ indicates that all other terms are combinations of $d\rho_j$, $\rho_j^3d\theta_j$ and $dy$ (where $y$ is a local
coordinate along $D_J$) with coefficients vanishing at this rate.  This (and in fact a slightly sharper version)
is proved in \cite{Wol2}, and the same result with some further information on the first derivatives of the
metric components appears in \cite{LSY}.  

We do not belabor the precise form of the remainders in these asymptotics for the following two reasons.  
First, $g_{\WP}$ is K\"ahler, and we can invoke the standard fact in K\"ahler geometry, see \cite[p.252]{Aub}, that if
$g$ is any K\"ahler metric, then its Laplace operator has the particularly simple form 
\begin{equation}
\Delta_{g} = \sum g^{i \bar{\jmath}} \frac{\partial^2\,}{\partial z_i \partial \overline{z_j}}, 
\label{lapkahler}
\end{equation}
with no first order terms. In particular, the coefficients of this operator do not depend on derivatives of the metric. The
same is true for the Dirichlet form for this metric, which also involves only the components of 
the (co)metric, but not their derivatives.  The proof in \S 3 involves various integrations by parts, but a 
close examination of the details shows that one needs in any case very little about the derivatives 
of the metric, and for K\"ahler metrics one needs no information about these derivatives at all.

The other reason is that current work by the second author and J.\ Swoboda aims at deriving a complete 
asymptotic expansion for $g_{\WP}$, and this implies all the results needed here about the remainder terms 
and their derivatives. However, since that work has not yet appeared, we emphasize that enough is known 
about the asymptotics of the metric in the existing literature to justify all the calculations below. 

In any case, using \eqref{gwp1}, disregarding the constants $\pi^3$ and $4$ for simplicity, 
and using the product polar coordinates, $(\rho_1,\ldots, \rho_\ell, \theta_1, \ldots, \theta_\ell, y_1, \ldots, y_s) \in 
\calU = (0, \rho_0)^\ell \times (S^1)^\ell \times  \mathcal V$, where $\calV$ is an open neighbourhood in $D_J$,
we note that 
\[
\|u\|^2_{L^2}=\int |u|^2  \mathcal J \cdot (\rho_1 \ldots \rho_\ell)^3 \,d\rho \, d\theta \, dy, 
\]
where the Jacobian factor $\mathcal J$ is uniformly bounded and uniformly positive. Similarly, the Dirichlet form 
for $\Delta_{\WP}$ is given by 
\[
\int \left(\sum_{j=1}^\ell |\pa_{\rho_j}u|^2+\sum_{j=1}^k \rho_j^{-6}|\pa_{\theta_j}u|^2
+|\nabla_y u|^2\right)  \mathcal J \cdot (\rho_1 \ldots \rho_\ell)^3 \, d\rho \, d\theta \, dy
\]
modulo terms vanishing like $|\rho|^3$.  

In order to focus on the key analytic points of the argument, we shall work with a slightly more general 
class of (not necessarily K\"ahler) Riemannian metrics, the asymptotic structure of which models that 
of $g_{\WP}$, with singularities of similar crossing cusp-edge type.  We thus let $\calM$ be any manifold
which has a compactification $\overline{\calM}$ with the same structural features as the Deligne-Mumford
compactification. Specifically, $\overline{\calM}$ is a stratified pseudomanifold, with $\calM$ its dense
top-dimensional stratum. All other strata are of even codimension, and can be locally described as finite
intersections $D_{j_1} \cap \ldots \cap D_{j_\ell}$, where $\dim D_j = \dim \calM - 2$ for all $j$. The main
point is that we can use the same sort of product polar coordinate systems as above, and we shall do
so henceforth without comment.  We now consider metrics which in any such local coordinate system 
are modelled by the product metric
\begin{equation}
g_{\ell, k} := \sum_{i=1}^\ell (d\rho_i^2+\rho_i^{2k}d\theta_i^2) + \sum_{\mu = 1}^{n - 2\ell} dy_\mu^2,
\label{cuspedgek} 
\end{equation}
for any $k\geq 3$. The corresponding Laplacian is given by
\begin{equation}
\Delta_{\ell,k} = - \sum_{i=1}^\ell  (\del_{\rho_i}^2 + \frac{k}{\rho_i} \del_{\rho_i} + \frac{1}{\rho_i^{2k}} \del_{\theta_i}^2) + \Delta_y = \Delta_{\ell,k}^{\perp} + \Delta_y.
\label{celap}
\end{equation} 
The first term on the right is the normal component of this model Laplacian. 
Note that we are restricting to $k \geq 3$. (We could even choose different orders $k_j$ on the 
different $D_j$, but for the sake of simplicity do not do so.) As we shall see below, the case $k = 3$ is in 
some sense a critical value, and the analysis is slightly easier for larger values of $k$. One motivation for 
discussing this more general setting is to clarify the borderline nature of the Weil-Peterson metric.

The main point is to clarify exactly what sorts of perturbations are allowed.  We phrase this by focusing
on the end result, i.e.\ by delineating the properties of the operators for which our arguments work,
and then `backfilling' by defining the corresponding class of metrics appropriately.  Thus we first assume that
\begin{equation}
\Delta_g= 
- \sum_{i=1}^\ell( \pa_{\rho_i}^2+k\rho_i^{-1}\pa_{\rho_i}+\rho_i^{-2k}\pa_{\theta_i}^2) + \Delta_{D_J} + E,
\label{celap2}
\end{equation}
where $\Delta_{D_J}$ is the Laplacian for the induced metric on the codimension $2\ell$ stratum,
and where $E$ is an error term.  The key structural assumptions are now as follows.  First, if
$f$ is supported (or indeed just defined since differential operators
are local) in one of these local coordinate systems and depends only on the $\rho_i$, then
we shall assume that 
\begin{equation}
\Delta_g f= - \sum_{i=1}^\ell( \pa_{\rho_i}^2+(k+ a_i)\rho_i^{-1}\pa_{\rho_i}) f +\sum a_{ij}\pa^2_{\rho_i \rho_j }f,
\label{celap3}
\end{equation}
where
\[
|a_{ij}|,\ |a_i|\leq C|\rho|^{\eta}
\]
for some $\eta > 0$ with $|\rho|=(\rho_1^2 +\ldots+\rho_{\ell}^2)^{1/2}$ the Euclidean length;
moreover, if $f$ depends only on $\rho$ and $y$, but not $\theta$, then 
\begin{equation}\begin{aligned}
\Delta_g & f= - \sum_{i=1}^\ell( \pa_{\rho_i}^2+(k+a_i)\rho_i^{-1}\pa_{\rho_i}) f + (\Delta_{D_J} +\sum c_{ij}\pa^2_{y_i y_j} )f \\
& +\sum a_{ij}\pa^2_{\rho_i\rho_j}f + \sum b_{ij}\pa^2_{\rho_i y_j} +\sum \tilde b_{ij}\rho_i^{-1}\pa_{y_j},
\end{aligned}
\label{celap4}
\end{equation}
where 
\[
|a_{ij}|,\ |a_i|,\ |b_{ij}|,\ |\tilde b_{ij}|,\ |c_{ij}|\leq C|\rho|^{\eta},
\]
again for some $\eta  > 0$. 

If $g$ is a K\"ahler metric and the coordinates $(\rho,\theta,y)$ are adapted to the complex structure, then
using \eqref{lapkahler} it is easy to guarantee that \eqref{celap3} and \eqref{celap4} hold simply by requiring that 
\begin{equation}
|g_{i \bar{\jmath}} - \delta_{ij}|, |g_{i \bar{\mu}}|, |g_{\mu \bar{\nu}} - (g_{D_J})_{\mu \bar{\nu}}|  \leq|\rho|^\eta, 
\label{decaykahler}
\end{equation}
for $i, j = 1, \ldots, \ell,\  \mu, \nu = \ell+1, \ldots, n$. More generally, if $g$ is a real (non-Hermitian) metric, 
assume the convention that $i, j, \ldots$ are indices for the 
$\rho$ and $\theta$ variables, and $\mu, \nu, \ldots$ are indices for the $y$ variables. Now write
\begin{equation*}\begin{aligned}
g=&\sum g_{ij}^{\rho \rho} d\rho_i d\rho_j + \sum g_{ij}^{\rho \theta} d\rho_i \rho_j^k d\theta_j +\sum g_{i j }^{\theta \theta} 
\rho_i^k d\theta_i \rho_j^k d\theta_j\\
&+\sum g_{\mu \nu }^{yy} dy_\mu dy_\nu+\sum g_{i\mu }^{\rho y} d\rho_i dy_\mu +\sum g_{i \mu}^{\theta y}\rho_i^k d\theta_i dy_\mu. 
\end{aligned}\end{equation*}
(The superscripts $\rho$, $\theta$ and $y$ have been affixed to the metric components because of
obvious ambiguities in the indices.) We require first that 
\begin{equation}
\label{eq:metric-terms}
\begin{aligned}
|g_{ij}^{\rho \rho} -\delta_{ij}|, \ |g_{ij}^{\rho \theta}|,  \ |g_{ij}^{\theta \theta} -\delta_{ij}|, & \ |g_{\mu \nu}^{yy}-(g_{D_J})_{\mu \nu}|, \\
& |g_{i \mu}^{\rho y}|, \ |g_{j \mu}^{\theta y}| \leq C |\rho|^\eta
\end{aligned}
\end{equation}
for all choices of indices and for some $\eta > 0$.  These conditions are sufficient to guarantee that all the
coefficients of the second order terms in \eqref{celap3} and \eqref{celap4} have the right form.  To control
the coefficients of the first order terms, we must impose some conditions on the first 
derivatives of certain components of the metric. To specify these, recall the standard formula for the Laplacian in
an arbitrary coordinate system $(w_1, \ldots, w_{2n})$,
\[
-\Delta_g = \sum g^{\alpha \beta} \del^2_{w_\alpha w_\beta} + \sum \left( \frac12 g^{\alpha \beta} \del_{w_\alpha} \log \det (g_{\alpha \beta}) +
\del_{w_\alpha} g^{\alpha \beta}\right) \del_{w_\beta}.
\]
Write
\[
\log \det g = 2k \sum \log \rho_i + A.
\]
Then the coefficient of $\del_{\rho_i}$ is
\begin{multline*}
\sum_j \left((g^{\rho\rho})^{ij} (k \rho_j^{-1} + \del_{\rho_j} A + \del_{\rho_j} (g^{\rho \rho})^{ij} \right) + \\
\sum_j \left( \rho_j^{-2k}(g^{\rho \theta})^{ji} \del_{\theta_j} A + \rho_j^{-2k} \del_{\theta_j} (g^{\rho \theta})^{ji}\right) + \\
\sum_\mu \left( (g^{\rho y})^{\mu i} \del_{y_\mu}A + \del_{y_\mu} (g^{\rho y})^{\mu i}\right), 
\end{multline*}
and the coefficient of $\del_{y_\mu}$ is
\begin{multline*}
\sum_j \left((g^{\rho y})^{j \mu } (k \rho_j^{-1} + \del_{\rho_j} A + \del_{\rho_j} (g^{\rho y})^{j \mu } \right) + \\
\sum_j \left( \rho_j^{-2k}(g^{\theta y})^{j \mu} \del_{\theta_j} A + \rho_j^{-2k} \del_{\theta_j} (g^{\theta y})^{j\mu}\right) + \\
\sum_\nu \left( (g^{y y})^{\nu \mu} \del_{y_\nu}A + \del_{y_\nu} (g^{y y})^{\nu \mu}\right).
\end{multline*}
Comparing with \eqref{eq:metric-terms}, we see that the new conditions we must impose are that
\begin{equation}
\begin{split}
|\del_{\rho_j} A|, &|\del_{\theta_j} A|, |\del_{y_\mu} A|, |\del_{\rho_j} (g^{\rho \rho})^{ij}|, |\rho_j^{-2k} \del_{\theta_j} (g^{\rho \theta})^{ji}|, 
|\del_{y_\mu} (g^{\rho y})^{\mu i}|, \\ & |\del_{\rho_j} (g^{\rho y})^{j \mu }|, 
|\rho_j^{-2k} \del_{\theta_j} (g^{\theta y})^{j\mu}|, |\del_{y_\nu} (g^{y y})^{\nu \mu}| \leq |\rho|^\eta.
\end{split}
\label{newconditions}
\end{equation}

We have written this out in some detail to indicate that it is possible to phrase the necessary conditions
in terms of metric components. However, it is clearly far simpler to think of these conditions using
\eqref{celap3} and \eqref{celap4}.

\medskip

We conclude this section with the following observation. 
Since most of the basic arguments in the remainder of this paper are local, we take this opportunity to note
the existence of a partition of unity $\{\psi_\alpha\}_{\alpha \in A}$ on $\calM_\gamma$ with the property that
each $\psi_\gamma$ is supported either away from all of the divisors or else on one of the product polar
coordinate charts above, and which satisfy $|\nabla\psi_\gamma|, |\Delta\psi_\gamma| \leq C$. 
Indeed, we need only choose these functions so that $\del_{\theta_j} \psi_\alpha$, $\pa_{\rho_i}$, $\rho_i^{-k}\pa_{\theta_i}$ 
and $\pa_{y_j}$, for $y \in D_J$, applied to this function are $\calO(\rho_j^{2k})$; one can even arrange that $\psi_\alpha$ is
independent of $\theta_j$ when $\rho_j$ is sufficiently small.

\section{Essential self-adjointness of the Weil-Petersson Laplacian}
The operator $\Delta_g$ is symmetric on $\calC^\infty_{0} (\calM_\gamma)$, but since this space 
is incomplete, we must consider the possibility that there is not a unique self-adjoint extension.
Because $\Delta_g$ is semibounded, there is always at least one, namely the Friedrichs extension. 
Whether there are others besides this depends on the following considerations.

Recall the general definition of the minimal and maximal domains of the Laplacian:
\[ 
\mathcal D_{\max}(\Delta) = \{ u \in L^2(\calM_\gamma; dV_{\WP}): \Delta u \in L^2(\calM_\gamma; dV_{\WP}) \}
\]
and
\[
\begin{split}
\mathcal D_{\min}(\Delta) =  \{ u \in & L^2(\calM_\gamma; dV_{\WP}): \exists\, u_j \in \calC^\infty_0(\calM_\gamma)\ \mbox{such
that}\ \\
& u_j \to u\ \mbox{and}\ \Delta u_j \to f\ \mbox{in}\  L^2(\calM_\gamma; dV_{\WP})\}. 
\end{split}
\]
The operator $\Delta_{\WP}$ is called essentially self-adjoint provided $\mathcal D_{\min}(\Delta) = \mathcal D_{\max}(\Delta)$,
and in this case this is the unique self-adjoint extension of $\Delta$ from the core domain $\calC^\infty_{0}$.  If
these subspaces are not equal, then the self-adjoint extensions are in bijective correspondence with the
Lagrangian subspaces of $\mathcal D_{\max}/\mathcal D_{\min}$ with respect to a natural symplectic structure
(coming from the classical Green identities).  Choosing such an extension is tantamount to specifying a boundary
condition. 

We prove here the
\begin{thm}
The scalar Laplace operator $\Delta_{\WP}$ is essentially self-adjoint on $L^2(\calM_\gamma, dV_{\WP})$. 
\end{thm} 
The rest of this section is devoted to  the proof. We also obtain the
\begin{cor}
This unique self-adjoint extension has discrete spectrum. 
\label{discspec}
\end{cor}

The key to proving these statements is to show that if $u \in \mathcal D_{\max}(\Delta)$, then $u$ must
decay sufficiently and have enough regularity to lie in $\mathcal D_{\min}(\Delta)$.  We accomplish this
here through a sequence of one- and multi-dimensional Hardy and interpolation estimates. 

\subsection{Hardy inequalities}
\begin{lem}\label{lem-hardy} 
Fix any measure space $(Y, d\nu)$ and consider the measure space $(X, d\mu)$ where $X = \RR^+ \times Y$ 
and $d\mu=\rho^\alpha d\rho\, d\nu$. Suppose that $u\in L^2_{\loc}(X)$ has support in $\{\rho < \rho_0\}$ and
satisfies $\rho^\beta\pa_\rho u\in L^2(X; d\mu)$. If $2\beta+\alpha>1$, then $\rho^{\beta-1}u \in L^2(X, d\mu)$ and
\begin{equation}
\|\rho^{\beta-1}u\|_{L^2(X; d\mu)}\leq \frac{2}{2\beta+\alpha-1}\|\rho^\beta\pa_\rho u\|_{L^2(X, d\mu)}.
\label{Hardy1} 
\end{equation}
If we drop the condition that $u$ is supported in a finite strip in $\RR^+$, then for any $\rho_0, \ep > 0$, 
there exists some $C>0$ such that
\begin{equation*}
\begin{split}
&\|\rho^{\beta-1}u\|_{L^2((0,\rho_0)\times Y ; d\mu)}\\ 
&\qquad\leq \frac{2}{2\beta+\alpha-1}\|\rho^\beta\pa_\rho u\|_{L^2((0,\rho_0+\ep)\times Y; d\mu)}
+C\|u\|_{L^2((\rho_0,\rho_0+\ep)\times Y; d\mu)}.
\end{split}
\end{equation*}
\end{lem}
\begin{proof}
Assume first that $u$ is supported in $\rho<\rho_0$. Choose a function $\phi\in\CI(\RR)$ which is nonnegative and
monotone nondecreasing, vanishes for $\rho \leq 1/2$ and with $\phi(\rho)=1$ for $\rho \geq 3/4$. We use $\phi(\rho/\delta)$
as a cutoff, with $\delta \searrow 0$. 

By hypothesis, $u\in H^1_{\loc}$, and $\phi(\rho/\delta) u = 0$ near $\rho=0$, we calculate
\begin{equation*}
\begin{split}
&(2\beta+\alpha-1)\|\phi(\rho/\delta)\rho^{\beta-1}u\|^2_{L^2(X; d\mu)} \\
&=(2\beta+\alpha-1)\int \phi(\rho/\delta)^2\rho^{2\beta+\alpha-2} u \overline{u}\,d\rho d\nu\\
&\leq \int \left((2\beta+\alpha-1)\phi(\rho/\delta)^2\rho^{2\beta+\alpha-2}+2\delta^{-1}\phi(\rho/\delta)\phi'(\rho/\delta)\rho^{2\beta+\alpha-1}\right) u\, \overline{u}\,d\rho dz\\
&=\int\pa_\rho(\phi(\rho/\delta)^2\rho^{2\beta+\alpha-1}) u\, \overline{u}\,d\rho dz\\
&=\int (\pa_\rho (\phi(\rho/\delta)^2\rho^{2\beta+\alpha-1}u)-\phi(\rho/\delta)^2\rho^{2\beta+\alpha-1}\pa_\rho u)\,\overline{u}\,d\rho dz\\
&=-\int \left(\phi(\rho/\delta)^2\rho^{2\beta+\alpha-1}u \,\overline{\pa_\rho u}+\pa_\rho u\,\phi(\rho/\delta)^2\rho^{2\beta+\alpha-1}\overline{u}\right)\,d\rho dz\\
&=-\langle \phi(\rho/\delta)\rho^{\beta-1}u,\phi(\rho/\delta)\rho^\beta\pa_\rho u\rangle_{L^2} - 
\langle \phi(\rho/\delta)\rho^\beta\pa_\rho u,\phi(\rho/\delta)\rho^{\beta-1}u\rangle_{L^2}.
\end{split}
\end{equation*}
Using the Cauchy-Schwartz inequality, this yields
\begin{equation}\label{eq:Hardy-CS}
(2\beta+\alpha-1)\|\phi(\rho/\delta)\rho^{\beta-1}u\|^2\leq 2 \|\phi(\rho/\delta)\rho^{\beta-1}u\| 
\|\phi(\rho/\delta)\rho^\beta\pa_\rho u\|, 
\end{equation}
so dividing through by $\|\phi(\rho/\delta)\rho^{\beta-1}u\|$ gives
\begin{equation}\label{eq:reg-Hardy}
(2\beta+\alpha-1)\|\phi(\rho/\delta)\rho^{\beta-1}u\|\leq 2 \|\phi(\rho/\delta)\rho^\beta\pa_\rho u\|_{L^2}
\leq 2 \|\rho^\beta\pa_\rho u\|_{L^2}.
\end{equation}
This inequality is obvious if $\|\phi(\rho/\delta)\rho^{\beta-1}u\| = 0$.  Finally, let $\delta\to 0$ to get
that $\rho^{\beta-1}u\in L^2$ and that the estimate of the lemma holds.

If $u$ is not compactly supported then replace $\phi(\rho/\delta)$ by $\phi(\rho/\delta)\psi(\rho)$ where $\psi\in\CI(\RR)$
is nonnegative, equals $1$ for $\rho \leq \rho_0$ and is supported in $\rho \leq \rho_0 + \ep$. 
Apply the Hardy inequality above to $\psi u$ to get
\begin{equation}\label{eq:Hardy-mod}
\frac{2\beta+\alpha-1}{2}\|\rho^{\beta-1}\psi(\rho)u\|_{L^2}
\leq \|\rho^\beta\psi \pa_\rho u\|_{L^2}+\|\rho^\beta\psi'(\rho) u\|_{L^2};
\end{equation}
in view of the support properties of $\psi$ and $\psi'$, this proves the lemma.
\end{proof}

Now suppose we are near an intersection of divisors $D_J$.  Choose coordinates as before, and
for simplicity, for any multi-indices $\sigma \in \RR^\ell$ and $\gamma \in \mathbb N^\ell$,  write 
$$
\rho^\sigma = \rho_1^{\sigma_1} \ldots \rho_\ell^{\sigma_\ell}, \quad \mbox{and}\quad 
(\rho\pa_\rho)^\gamma= (\rho_1\del_{\rho_1} )^{\gamma_1} \ldots (\rho_\ell \del_{\rho_\ell})^{\gamma_\ell}.
$$
If $s \in \RR$, then we also write $\rho^s = (\rho_1 \ldots \rho_\ell)^s$. We also define
$$
\langle \rho \rangle =\left(\sum\rho_i^{-1}\right)^{-1}
$$
so $\langle \rho \rangle = 0$ when any of the $\rho_i$ vanishes. Note that $\langle \rho \rangle \leq\rho_j$ for any $j$.
Then the one-dimensional Hardy inequality above immediately gives

\begin{lem} Let $X = (\RR^+)^\ell \times Y$ where $(Y,d\nu)$ is a measure space and set
$d\mu=\rho^\alpha d\rho\, d\nu $. Fix $\alpha,\beta\in\RR^\ell$ such that $2\beta_i+\alpha_i>1$
for each $i$. If $u\in L^2_{\loc}(X, d\mu)$ is supported in $\{\rho_i<\rho_0\ \forall i\}$ and
$\rho_i^\beta\pa_{\rho_i} u\in L^2(X; d\mu)$ for each $i$, then 
$$
\|\rho^{\beta}\rho_i^{-1}u\|_{L^2(X; d\mu)}\leq \frac{2}{2\beta_i+\alpha_i-1}\|\rho^\beta\pa_{\rho_i} u\|_{L^2(X; d\mu)}.
$$

If we do not assume that $u$ has compact support, then for any $\rho_0,\ep>0$, there 
exists $C>0$ such that 
\begin{equation*}
\begin{split}
&\|\rho^{\beta}\rho_i^{-1}u\|_{L^2(X \cap \{\rho_i < \rho_0\}; d\mu)}\\
&\qquad\leq \frac{2}{2\beta_i+\alpha_i-1}\|\rho^\beta\pa_{\rho_i}
u\|_{L^2(X \cap \{\rho_i \rho_0+\ep\}; d\mu)} +C\|\rho^\beta u\|_{L^2(X \cap \{\rho_0 \leq \rho_i \leq \rho_0+\ep\}; d\mu)}.
\end{split}
\end{equation*}
\end{lem}

\subsection{Interpolation inequalities}
For any $\sigma \in\RR^\ell$, define the space
\begin{equation}
X^{-\sigma}=\{u \in L^2_{\mathrm{loc}}: 
\|u\|^2_{X^{-\sigma}} :=\|\rho^\sigma\langle \rho \rangle ^{-1}u\|^2+\|\rho^\sigma \langle \rho \rangle \Delta u\|^2 < \infty\}.
\label{Xsig}
\end{equation}
We write $X^{-\sigma}(\calU)$ for the space of functions with finite $X^{-\sigma}$ norm in $\calU$. 

Our first task is to show that the $X^{-\sigma}$ norm 
controls the $L^2$ norm of $\rho^\sigma \nabla u$.   Note that $\rho^\sigma \langle \rho \rangle^{-1} u \in L^2$ 
is equivalent to $\rho^\sigma\rho_j^{-1}u\in L^2$ for all $j$.   This result is local, so we fix $\rho_0 > 0$ 
and work in a neighbourhood $\calU_{\rho_0}=\{\rho_j<\rho_0,\ j=1,\ldots, k\}$.

\begin{lem}\label{lem:interpolate}
Let $\sigma \in\RR^\ell$, $\ep>0$, and suppose that $u \in X^{-\sigma}(\calU_{\rho_0 + \ep})$. Then
$\rho^\sigma \nabla u$ is in $L^2(\cU_{\rho_0})$ and
\[
\|\rho^\sigma \nabla u\|^2_{L^2(\cU_{\rho_0})}\leq C\|\rho^\sigma \langle \rho\rangle^{-1}u\|_{  L^2(\cU_{\rho_0+\ep})}
\, ||u||_{X^{-\sigma}(\calU_{\rho_0+\ep})} 
\]
In particular 
\[
\rho \langle \rho \rangle ^{-1}u, \ \rho \langle \rho \rangle \Delta u\in L^2 \Longrightarrow \rho\nabla u\in L^2.
\]
\end{lem}

\begin{rem}
As a consequence, if $u\in X^{-\sigma}$, then $\rho^\sigma\nabla u\in L^2$. Furthermore, if $\sigma_i\geq 1$ for all $i$ then
$u\in L^2$ implies $\rho^\sigma \langle \rho \rangle^{-1}u\in L^2$.  In particular, if $u, \Delta u \in L^2$,
then $u\in X^{-\sigma}$ for all $\sigma$ with $\sigma_i\geq 1$ for all $i$.
\end{rem}

\begin{proof}
Let $f$ be real-valued and $\calC^\infty$ and denote by $m_f$ the operator of multiplication by $f$.  We claim that there 
is an equality of differential operators
\begin{equation}
\Delta \circ m_f+ m_f \circ \Delta=2\nabla^* ( m_f \circ \nabla \,) + m_{\Delta f} .
\label{mfop}
\end{equation}
Indeed, both sides are symmetric and have the same principal symbol, so their difference is a symmetric, real, first order 
scalar operator, hence must actually have order zero. Applying both sides to the constant function $1$ gives the claim. 

Suppose that $f\geq 0$ is real valued and has compact support in $\calM_\gamma$. Applying the operator in \eqref{mfop} 
to $u$ and taking the inner product with $u$ gives
\begin{equation}\label{eq:pair-symm-Delta-f}
2 \mbox{Re}\, \langle \Delta u, fu\rangle =2\|f^{1/2}\nabla u\|^2+\langle (\Delta f)u,u\rangle.
\end{equation}
This equality extends to all $u\in H^2_{\loc}(\calM_\gamma)$. 

Choose $\phi, \psi \in \calC^\infty(\RR)$ with $0 \leq \phi, \psi \leq 1$ and such that 
\[
\phi(t) = 0\ \mbox{for}\ t \leq 1/2, \phi(t) = 1\ \mbox{for}\  t \geq 1,
\]
\[
\psi = 1\ \mbox{for}\ t \leq \rho_0,\ \psi(t) = 0\ \mbox{for}\ t \geq \rho_0 + \frac{\ep}{2}. 
\]
Furthermore, let $\chi\in\CI_c(\RR^s)$ be supported in the $y$-coordinate chart $\calV$. Then define
\[
f_{i,\delta} = \phi(\rho_i/\delta)^2\psi(\rho_i)^2\rho_i^{2\sigma_i}, \qquad 
f_\delta= f_{1,\delta} \ldots f_{\ell, \delta}\chi(y)^2. 
\]
In order to use this in the identity above, we must compute 
\begin{equation*}\begin{aligned}
\Delta
f_\delta&=\left(\sum_{i=1}^\ell (-\pa_{\rho_i}^2-k\rho_i^{-1}\pa_{\rho_i}) -E + \Delta_{D_J}\right)f_\delta \\
&=\sum_{i=1}^\ell
\rho_i^{-2}\Big(-(\rho_i\pa_{\rho_i})^2-(k-1+a_i-(\sum_ja_{ij}))(\rho_i\pa_{\rho_i})\Big)f_{\delta}\\
&\qquad-\sum_{i,j=1}^{\ell}\rho_i^{-1}\rho_j^{-1}a_{ij}(\rho_i\pa_{\rho_i})(\rho_j\pa_{\rho_j})f_\delta+\Delta_{D_J}f_\delta\\
&\qquad-\sum \rho_i^{-1}b_{ij}(\rho_i\pa_{\rho_i})\pa_{y_j}f_\delta-\sum
c_{ij}\pa_{y_i}\pa_{y_j}f_\delta
-\sum \tilde b_{ij}\rho_i^{-1}\pa_{y_j}f_\delta.
\end{aligned}\end{equation*}
Then, since 
\[
(\rho_i\pa_{\rho_i})^j (\phi(\rho_i/\delta))=((t\del_t)^j\phi)(\rho_i/\delta),\ j = 1, 2, 
\]
we see that $\langle\rho\rangle^2\rho^{-2\sigma} \Delta f_\delta$ is uniformly bounded as $\delta \searrow 0$ (recall
that $\rho_i\geq\langle\rho\rangle$ for all $i$).   Hence, setting $f = f_\delta$ in \eqref{eq:pair-symm-Delta-f}, then 
\[
|\langle (\Delta f)u,u\rangle| \leq C \|\rho^{\sigma}\langle\rho\rangle^{-1}u\|^2
\]
uniformly in $\delta$. Applying Cauchy-Schwartz to \eqref{eq:pair-symm-Delta-f} thus shows that
\[
2\|\prod_{i=1}^\ell\phi(\rho_i/\delta)\psi(\rho_i))\chi(y)\rho^\sigma\nabla u \|^2\leq C\left(\|\rho^\sigma \langle\rho\rangle^{-1}u\|^2+
\|\rho^\sigma\langle\rho\rangle^{-1}u\|\|\rho^\sigma\langle\rho\rangle\Delta u\|\right),
\]
and as $f_\delta \to f_0 = \psi(\rho_1)^2 \ldots \psi(\rho_\ell)^2 \chi(y)^2\rho^{2\sigma}$ as $\delta \to 0$ we conclude that $\psi(\rho_1)\ldots \psi(\rho_\ell) \rho^\sigma \nabla u \in L^2$ by letting $\delta\to 0$. 
\end{proof}

Another useful property of these spaces is that they localize. 
\begin{lem}\label{lemma:X-r-loc}
Let $\{\psi_\alpha\}$ be the partition of unity described at the end of \S 2. Then $u\in X^{-\sigma}$
implies $\psi_\alpha u\in X^{-\sigma}$.
\end{lem}
\begin{proof}
First note that 
$$
\Delta(\psi_\alpha u)=\psi_\alpha\Delta u-2(\nabla \psi_\alpha,\nabla u)_g+(\Delta \psi_\alpha)u.
$$
The expression in the middle is the pointwise inner product with respect to the metric $g$. Since
$\psi_\alpha$, $\langle \rho \rangle \nabla\psi_\alpha$ and $\langle \rho \rangle ^2\Delta \psi_\alpha$
are all bounded (even without the factor of $\langle \rho \rangle$) we obtain 
$\rho^\sigma \langle \rho \rangle \Delta(\psi_\alpha u)\in L^2$. However, $\rho^\sigma \langle \rho\rangle ^{-1}\psi_\alpha
u\in L^2$ as well, so $\psi_\alpha u\in X^{-\sigma}$.
\end{proof}

\begin{cor}\label{cor:X-r-dense}
For any $\sigma \in \RR^\ell$, $\calC^\infty_0(\calM_\gamma)$ is dense in $X^{-\sigma}$. 
\end{cor}
\begin{proof}
Fix any $u\in X^{-\sigma}$. By the previous Lemma, we may as well assume that 
$u$ is supported in $\calU_{\rho_0}$. Lemma~\ref{lem:interpolate} implies that $\rho^\sigma\nabla u\in L^2$. 

Choose the cutoff function $\phi(t)$ as in the proof of Lemma~\ref{lem:interpolate}
and set $\Phi_\delta(\rho) =\phi(\rho_1/\delta) \ldots \phi(\rho_\ell/\delta)$. 
The function $\Phi_\delta u$ is compactly supported in $\calM_\gamma$, and
$\Delta u \in L^2_{\mathrm{loc}}$, so $\Phi_\delta u\in H^2_0(\calM_\gamma)$, and
hence clearly lies in $X^{-\sigma}$. 

We claim that $\Phi_\delta u\to u$ in $X^{-\sigma}$, i.e. 
\[
\rho^{\sigma} \langle \rho \rangle ^{-1}\Phi_\delta u\to \rho^{\sigma}\langle \rho\rangle ^{-1}u,
\quad \mbox{and} \quad \rho^{\sigma }\langle\rho\rangle\Delta(\Phi_\delta u)\to \rho^{\sigma}
\langle\rho\rangle\Delta u
\]
in $L^2$. The former follows from the dominated convergence
theorem. For the latter, we use that
$$
\Delta (\Phi_\delta u)=\Phi_\delta \Delta u-2(\nabla\Phi_\delta,\nabla
u)_g+(\Delta\Phi_\delta)u.
$$
Since $\rho^{\sigma }\langle \rho \rangle \Delta u\in L^2$, dominated convergence gives that 
$\rho^{\sigma}\langle\rho\rangle \Phi_\delta\Delta u\to\rho^{\sigma}\langle\rho\rangle \Delta u$.
In addition $\rho^\sigma\nabla u\in L^2$ and $\rho^{\sigma}\langle \rho \rangle^{-1}u\in L^2$, 
so the estimate for the remaining term follows from the fact that 
$\langle \rho \rangle\nabla\Phi_\delta$ and $\langle \rho\rangle^2\Delta
\Phi_\delta$ are uniformly bounded and converge to $0$ pointwise.  
These bounds on $\Phi_\delta$ follow from $(\rho_i\pa_{\rho_i})^k
\phi(./\delta)=((t\del_t)^k\phi)(./\delta)$ (cf.\ the proof of Lemma~\ref{lem:interpolate}), and 
the pointwise convergence is clear. 
This completes the proof of the convergence claim, and the corollary follows from the density 
of $\calC^\infty_0(\calM_\gamma)$ in $H^2_0(\calM_\gamma)$.
\end{proof}

\subsection{Improving the decay rate}
The only reasonable general estimate for the weighted norm of $\nabla u$
involves the weighted norms of $\Delta u$ and on $u$, with closely related 
powers of the weight function.  However, there is a critical range of weights
in which one can estimate $\|\rho^r \nabla u\|_{L^2}$ using $\|\rho^{r} 
\langle \rho \rangle \Delta u\|_{L^2}$ but a much weaker norm of $u$. 
We explain this now, and in particular develop localized versions of these estimates.

\begin{lem}\label{lemma:Delta-X-r_0-loc}
Let $\psi_\alpha$ be the partition of unity defined at the end of \S 2.  Fix $\sigma, \sigma_0\in\RR^\ell$
such that $\sigma_0- \ell^{-1} \leq \sigma < \sigma_0$.  Let $u\in X^{-\sigma_0}$ satisfy
$\rho^{\sigma} \langle \rho \rangle \Delta u\in L^2$. Then $\rho^{\sigma} \langle\rho\rangle \Delta(\psi_\alpha 
u)\in L^2$.
\end{lem}

Note that the conclusion is, up to one order of decay, better than the assumption $u\in X^{-\sigma_0}$ which 
implies that $\rho^{\sigma_0}\langle\rho\rangle\Delta (\psi_\alpha u)\in L^2$.

\begin{proof}
Expanding 
\[
\Delta(\psi_\alpha u)=\psi_\alpha\Delta u-2(\nabla \psi_\alpha, \nabla u)_g +(\Delta \psi_\alpha)u,
\]
and using that $\rho^\sigma \langle \rho \rangle \Delta u \in L^2$ and $\nabla \psi_\alpha$ and $\Delta \psi_\alpha$
are bounded, we see that it suffices to check that
\[
\rho^{\sigma - \sigma_0} \langle \rho \rangle \leq C.
\]
However, this follows since $\sigma_j - (\sigma_0)_j \geq -1/\ell$ for every $j$ and
\[
\langle \rho \rangle = ( \rho_1 \ldots \rho_\ell)^{1/\ell} R,
\]
where $R \in L^\infty$. 
\end{proof}

We now prove a stronger version of this, that a relatively weak weighted bound on $u$
and a stronger weighted bound on $\Delta u$ imply 
We can now prove that just information on the weighted space $\Delta
u$ sits in suffices to obtain information on $u$ and $\nabla u$ if at
least weak a priori weighted information on $u$ is available.

\begin{lem}\label{lemma:indicial-gap}
Let $-\frac{k-2}{4}<\sigma<\sigma _0 <\frac{k-1}{2}$, and suppose that $\rho^{\sigma_0}\langle \rho \rangle^{-1} u\in L^2$
and $\rho^{\sigma}\langle \rho \rangle\Delta u\in L^2$; then $u \in X^{-\sigma}$. 
\end{lem}
\begin{proof}
It suffices to prove the lemma assuming only that $\sigma \geq \sigma_0 - 1/\ell$, for once we have established
this, then we may iterate a finite number of times to obtain the result as stated. Furthermore, we 
can also replace $u$ by $\psi_\alpha u$, so as to work in a single coordinate chart. 

Recall from Lemma~\ref{lem:interpolate} that if $u \in X^{-\sigma_0'}$ for any $\sigma_0'$, then
$\rho^{\sigma'_0}\nabla u\in L^2$.   Now, consider the first and third terms in \eqref{eq:pair-symm-Delta-f},
and rewrite these as follows: 
\[
2\mbox{Re}\, \langle \Delta u, f u \rangle = 2 \mbox{Re}\, \big\langle \rho^{\sigma_0'} \langle \rho \rangle \Delta u, 
(\rho^{-2\sigma_0'} f)  \rho^{\sigma_0'}  \langle \rho \rangle^{-1} u \big\rangle
\]
and
\[
\langle (\Delta f) u, u \rangle = \big\langle ( \rho^{-2\sigma_0'} \langle \rho \rangle^2 \Delta f)  \rho^{\sigma_0'} \langle \rho 
\rangle^{-1} u, \rho^{\sigma_0'} \langle \rho \rangle^{-1} u \big\rangle. 
\]
From this it follows that if $f$ is smooth but not necessarily compactly supported, and
\begin{equation}
\rho^{-2\sigma'_0}f, \ \rho^{-2\sigma'_0}\langle \rho \rangle^2\Delta f \in L^\infty,
\label{bounds-f}
\end{equation}
then these two terms yield continuous bilinear forms on $X^{-\sigma'_0}$.   Hence, since
\eqref{eq:pair-symm-Delta-f} holds for all $u$ lying in the dense subspace $\calC^\infty_0(\calM_\gamma)$,
it extends by continuity to all of $X^{-\sigma_0'}$;  in other words, \eqref{eq:pair-symm-Delta-f} holds
for all elements of this space. 

The next step is a judicious choice of the function $f$, or rather, a family 
of such functions, satisfying these properties.  We define
$$
f=f_\delta= \Phi_\delta(\rho)^2 \rho^{2\sigma}, \quad \mbox{where} \quad
\phi_j(t)= \left(\frac{t}{1+t}\right)^{\beta_j},\qquad\beta_j>0,
$$
where, as before, $\Phi_\delta$ is the product of the functions $\phi_j(\rho_j/\delta)$ over $j = 1, \ldots, \ell$. 
It is clear that both conditions in \eqref{bounds-f} hold provided $\sigma+\beta\geq \sigma'_0$. 

We claim now that 
\begin{equation}\label{eq:Lap-of-f}
\Delta f_\delta+\sum_{i=1}^\ell C_i\rho_i^{-2}f_\delta\geq 0\quad \mbox{for some}\ C_i<\frac{(2\sigma_i+k-1)^2}{2}.
\end{equation}
This will be established below.  For later use, fix $b_1 >0$ and $b_2 \in(0,2)$ such that
\begin{equation}
C_i+ b_1 < b_2 \frac{(2\sigma_i+k-1)^2}{4} \ \Leftrightarrow \ 
-b_2 \frac{(2\sigma_i+k-1)^2}{4}  + b_1 < -C_i
\label{consts}
\end{equation}
for all $i = 1, \ldots, \ell$.  

Since the functions $\phi_j$ defined here satisfy $\phi_j'\geq 0$, the proof of the {\em regularized} Hardy inequality 
\eqref{eq:reg-Hardy} carries through exactly as before, giving
\begin{equation*}
\sum_{i=1}^\ell \frac{(2\sigma_i+k-1)^2}{4} \|\Phi_\delta(\rho) \rho^{\sigma}\rho_i^{-1}u\|^2 \leq 
\sum_{i=1}^\ell \| \Phi_\delta(\rho) \rho^{\sigma}\pa_{\rho_i} u\|^2  \leq 
\| \Phi_\delta(\rho) \rho^{\sigma}\nabla u\|^2. 
\end{equation*}
Rearranging this and using \eqref{consts}, we deduce 
\begin{equation*}\begin{split}
(2-&b_2 )\| \Phi_\delta(\rho) \rho^{\sigma}\nabla u\|^2 + b_1 \sum_{i=1}^\ell\| \Phi_\delta(\rho) \rho^{\sigma}\rho_i^{-1}u\|^2\\
&\leq 2\| \Phi_\delta(\rho) \rho^{\sigma}\nabla u\|^2 +\langle (\Delta f_\delta) u,u\rangle, 
\end{split}\end{equation*}
and hence, recognizing this as the right side of \eqref{eq:pair-symm-Delta-f} with our particular choice of $f$ and using
the Cauchy-Schwartz inequality, we bound this expression by 
\begin{equation*}\begin{split}
& 2\| \Phi_\delta(\rho) \rho^{\sigma}\langle \rho \rangle^{-1}u\|\|\Phi_\delta(\rho) \rho^{\sigma}\langle \rho \rangle\Delta u\|\\
&\leq \frac{b_1}{2\ell} \|\Phi_\delta(\rho) \rho^{\sigma}\langle \rho \rangle^{-1}u\|^2+
\frac{2\ell}{b_1}\|\Phi_\delta(\rho) \rho^{\sigma}\langle \rho \rangle\Delta u\|^2.
\end{split}\end{equation*}
Now, $\langle \rho \rangle^{-2}\leq \ell \sum_{i=1}^\ell\rho_i^{-2}$, so the first term on
the right hand side can be absorbed in the second term on the left hand side. Bounding this 
new term from below by exactly the same inequality, we obtain finally that
\begin{equation*}
(2-b_2)\|\Phi_\delta(\rho) \rho^{\sigma}\nabla u\|^2 +\frac{b_1}{2\ell }\|\Phi_\delta(\rho) \rho^{\sigma}
\langle \rho \rangle^{-1}u\|^2 \leq\frac{2\ell}{b_1}\| \Phi_\delta(\rho) \rho^{\sigma}\langle \rho \rangle\Delta u\|^2.
\end{equation*}
To conclude the argument, let $\delta\searrow 0$; this shows that $\rho^\sigma\nabla u$ and 
$\rho^{\sigma}\langle \rho \rangle^{-1}u$ lie in $L^2$, as desired.

It remains to show that \eqref{eq:Lap-of-f} holds for some
$C_i<\frac{(2\sigma_i+k-1)^2}{2}$. First,
\begin{equation*}\begin{aligned}
\Delta
f_\delta
&=\sum_{i=1}^\ell
\rho_i^{-2}\Big(-(\rho_i\pa_{\rho_i})^2-(k-1+a_i-(\sum_ja_{ij}))(\rho_i\pa_{\rho_i})\Big)f_{\delta}\\
&\qquad-\sum_{i,j=1}^{\ell}\rho_i^{-1}\rho_j^{-1}a_{ij}(\rho_i\pa_{\rho_i})(\rho_j\pa_{\rho_j})f_\delta
\end{aligned}\end{equation*}
with $|a_i|,|a_{ij}|\leq C|\rho|^{\eta}$.
A straightforward but somewhat lengthy computation gives that
\begin{multline*}
\Delta f_\delta+\sum_{i=1}^lC_i\rho_i^{-2}f_\delta= \\
-\Phi_\delta(\rho)^2 \rho^{2\sigma} \Big(\sum_{j=1}^\ell \rho_j^{-2}
g_j(\rho_j/\delta)+\sum_{i,j=1}^{\ell} \rho_i^{-1}\rho_j^{-1}\tilde h_{ij}(\rho_i/\delta,\rho_j/\delta)\Big).
\end{multline*}
Here
\begin{equation*}
\begin{split}
g_j(t) & = \phi_j(t)^{-2}\left( (t\del_t)^2 + (4\sigma_j + k-1)
  t\del_t + 2\sigma_j (2 \sigma_j + k-1) - C_j\right) \phi_j(t)^2  \\ 
& = \frac{2\beta_j( 2\beta_j + 1)}{(1+t)^2} + \frac{ (4\sigma_j + k-2) 2\beta_j}{1+t} + 2\sigma_j(2\sigma_j + k-1) - C_j
\end{split}
\end{equation*} 
results from the model part of the Laplacian, while $\tilde h_{ij}$ comes from applying the remainder terms
$a_i$ and $a_{ij}$, so that 
\[
|\rho|^{-\eta}|\tilde h_{ij}(s,t)|\leq\tilde C,\ s,t\geq 0
\]
for some $\tilde{C} > 0$. 

To do this computation efficiently, say for the $g_j$ term, note first that commuting the factor $\rho^{2\sigma}$ in $f$ through
each differential expression $\rho_i\pa_{\rho_i}$ replaces this expression by $\rho_i\pa_{\rho_i}+2\sigma_i$. Next, 
\[
(\rho_i\pa_{\rho_i}+2\sigma_i)(\phi_i(\rho_i/\delta)^2) =(t\pa_t+2\sigma_i)\phi_i^2|_{t=\rho_i/\delta},
\]
so we have reduced to computing the action of a second order regular singular ordinary differential operator
\[
-(t\pa_t+2\sigma_i)^2-(k-1)(t\pa_t+2\sigma_i).
\]
on the function $(t/(1+t))^{\beta_i}$, which is straightforward. 

Now, if $-g_j(t)\geq \delta>0$ for all $j$ and for some $C_j<\frac{(2\sigma_j+k-1)^2}{2}$, then using
$\rho_i^{-1}\rho_j^{-1}\leq\frac{1}{2}(\rho_i^{-2}+\rho_j^{-2})$, one can estimate
\[
|\sum_{i,j=1}^{\ell} \rho_i^{-1}\rho_j^{-1}\tilde h_{ij}(\rho_i/\delta,\rho_j/\delta)|\leq|\rho|^{\eta}\tilde
C\ell\sum_{i=1}^\ell \rho_i^{-2}.
\]
Consequently, assuming that $|\rho|$ is sufficiently small, which is possible here by adjusting the partition 
of unity, we deduce that \eqref{eq:Lap-of-f} holds.

We now wish to show that $g_j(t) \leq 0$ for $t \geq 0$. To this end, note that 
\[
g_j(0) = 2 (\beta_j + \sigma_j) (2\beta_j + 2 \sigma_j + k-1) - C_j,
\]
so this must certainly be nonpositive.   Furthermore, 
\[
\begin{split}
g_j'(t) & = - \frac{4\beta_j (2\beta_j + 1)}{(1+t)^3} - \frac{2\beta_j(4\sigma_j + k-2)}{(1+t)^2} \\
& = -\frac{2\beta_j}{(1+t)^3} \big( 2 (2\beta_j + 1) + (4 \sigma_j + k-2) (1+t) \big),
\end{split}
\]
so we wish that
\[
4\beta_j + 2 + (4\sigma_j + k-2)(1+t) \geq 0. 
\]
Since $4\sigma_j + k-2 > 0$, this is bounded below by $4(\beta_j + \sigma_j) + k$. 

Setting $\gamma_j = \beta_j + \sigma_j$, we have now shown that $g_j(t) \leq 0$ provided
\[
4\gamma_j + k \geq 0 \quad \mbox{and}\quad  2 \gamma_j (2\gamma_j + k-1) - C_j \leq 0. 
\]
Calculating the roots of this quadratic equation, we see that
\[
-\frac{k}{4} \leq \gamma_j \leq - \frac{k-1}{4} + \frac{1}{2}\sqrt{ \frac{(k-1)^2}{4} + C_j}
\]
implies $g_j(t) \leq 0$ for $t \geq 0$.
The leftmost term here is always less than the rightmost provided $C_j > 0$, so there is always a permissible
range for $\gamma_j$. Indeed, combining the various restrictions above, we see that we must choose
\begin{equation}
2 (\sigma_j + \beta_j) (2 (\sigma_j + \beta_j) + k-1) < C_j < \frac{ (2\sigma_j + k-1)^2}{2}. 
\label{late}
\end{equation}
Once again, setting $\beta_j = 0$, the leftmost term is less than the rightmost term precisely when
\[
- \frac{k-1}{2} < \sigma_j < \frac{k-1}{2},
\]
so there is certainly some allowable interval for $C_j$ provided $\beta_j$ is sufficiently small.

Since we wish to iterate the argument above, it is useful to estimate how large we can choose $\beta_j$
so that there still exists an admissible $C_j$. One half the
difference between the left and right hand side of \eqref{late} equals
\begin{equation*}
\begin{split}
&\frac{1}{2}\left(2(\sigma_j+\beta_j)(2(\sigma_j+\beta_j)+k-1)-\frac{(2\sigma_j+k-1)^2}{2}\right)\\
&\qquad=\Big(\sigma_j-\frac{k-1}{2}\Big) \Big(\sigma_j+\frac{k-1}{2}\Big)+4\beta_j \Big(\sigma_j+\frac{k-1}{4}\Big)+2\beta_j^2.
\end{split}\end{equation*}
Taking $\beta_j=\gamma(\frac{k-1}{2}-\sigma_j)$, which is positive provided $\gamma > 0$, gives 
$$
\Big(\sigma_j-\frac{k-1}{2}\Big) \Big\{\Big(\sigma_j+\frac{k-1}{2}\Big)-4\gamma\Big(\sigma_j+\frac{k-1}{4}\Big)+
2\gamma^2 \Big(\sigma_j-\frac{k-1}{2}\Big)\Big\}.
$$
Since $\sigma_j - \frac12 (k-1) < 0$, we need the other factor to be positive.  Let us write this factor as
$2A \gamma^2 - 4B \gamma + C$.  Thus
\[
A = \sigma_j - \frac{k-1}{2} < 0,\quad B = \sigma_j + \frac{k-1}{4} > \frac14 \quad \mbox{and}\ C = \sigma_j + \frac{k-1}{2} > \frac{k}{4}.
\]
These sign conditions and the quadratic formula show that there is one positive and one negative root of this quadratic equation,
and since the leading coefficient $2A$ is negative, we see that there is a $\gamma_0 > 0$ depending only on $k$ such that if 
$\gamma \in (0, \gamma_0)$, and with $\beta_j$ chosen as above, then there is indeed a gap between the left and right
sides of \eqref{late}.  In other words, 
we have now proved that if $-\frac{k-2}{4}<\sigma_j<\frac{k-1}{2}$ and if 
$\rho^{\sigma'_0}\langle \rho \rangle^{-1}u,\rho^{\sigma'_0}\langle \rho \rangle\Delta u\in L^2$,  with
$\sigma_0'=\sigma +\beta= \sigma+\gamma\Big(\frac{k-1}{2}-\sigma\Big)$, for any fixed 
$\gamma\in (0,\gamma_0)$, then $\rho^\sigma \nabla u$ and $\rho^{\sigma}\langle \rho \rangle^{-1}u$
both lie in $L^2$.    This shows that starting with $\sigma_0<\frac{k-1}{2}$, then we can iterate the
entire argument a finite number of times to conclude that $\rho^\sigma \nabla u,\rho^{\sigma}\langle \rho \rangle^{-1}u\in L^2$
for any $\sigma >-\frac{k-2}{4}$. 
\end{proof}

There is a variant of this result which holds even in the borderline case $\sigma_0= \frac{k-1}{2}$. 
We lose a bit, however, in that we can only estimate some combination of $\nabla u$ and
$\langle \rho \rangle^{-1}u$, but not these two terms separately.

\begin{lem}\label{lemma:indicial-bdy}
Set $\sigma_0=\frac{k-1}{2}$ and suppose that $\max(0,\sigma_0-1/\ell)\leq \sigma < \sigma_0$, and 
$\sigma \neq 0$. Suppose furthermore that the rate of decay $\eta$ for the error terms in the metric satisfies
$\eta \geq 1$ (note, this is certainly true for $g_{\WP}$).  If $\rho^{\sigma_0}\langle 
\rho \rangle^{-1}u\in L^2$ and $\rho^{2\sigma-\sigma_0}\langle \rho \rangle\Delta u\in L^2$ then $\rho^{\sigma-2\sigma_0}\nabla 
\rho^{2\sigma_0} u\in L^2$.
\end{lem}
\begin{proof}
By assumption, $u\in X^{-\sigma_0}$, and in view of Lemma~\ref{lemma:Delta-X-r_0-loc}, we may replace $u$
by some $\psi_\alpha u$ so as to assume that $u$ has compact support in some chart. 

We now claim that acting on $\CI_c(\calM_\gamma)$, one has
\begin{equation}\begin{split}\label{eq:recentered-est}
f\Delta+\Delta f= & \ 2(\rho^{-(k-1)}\nabla\rho^{k-1})^*f(\rho^{-(k-1)}\nabla\rho^{k-1})\\
&\qquad+\Delta f-2(\rho^{-(k-1)}\nabla\rho^{k-1})^*(f\rho^{-(k-1)}(\nabla\rho^{k-1})), 
\end{split}\end{equation}
where the two occurrences of $f$ on the left as well as the last two terms on the right are multiplication
operators. Indeed, both sides are formally self-adjoint, real, and have the same principal symbol, so
the difference is an operator of order $0$; evaluation on the constant function $1$ then gives the result. 
The last term on the right in \eqref{eq:recentered-est} can be rewritten in a more useful way as follows.
First write this term as $-2\rho^{k-1}\nabla^*\rho^{-2(k-1)}f\nabla\rho^{k-1}$, then commute to the 
front the middle factor $\rho^{-2(k-1)}f$; omitting the initial minus sign, this yields
\[
\begin{split}
\qquad \qquad 2\rho^{-(k-1)}f\Delta\rho^{k-1}-2\rho^{k-1}[\nabla,\rho^{-2(k-1)}f]^*\nabla\rho^{k-1} \qquad \qquad \qquad \\
= 2\rho^{-(k-1)}f\Delta\rho^{k-1} - 2\rho^{-(k-1)}\langle\nabla f,\nabla \rho^{k-1}\rangle_g - 
2f\rho^{k-1}\langle\nabla\rho^{-2(k-1)},\nabla\rho^{k-1}\rangle_g \\
= - 2(k-1)\rho^{-1}\langle \nabla \rho, \nabla f\rangle_g + 2\rho^{-(k-1)}f\Delta\rho^{k-1} + 4(k-1)f\rho^{-k}\langle \nabla 
\rho, \nabla \rho^{k-1}\rangle_g.
\end{split}
\]
Using \eqref{celap3}, we expand and then combine the last two terms on the right; this produces a cancellation,
from which all that remains is $f \sum_{i, j} a'_{ij} \rho_i^{-2}$, for some coefficients $a_{ij}'$ which satisfy $|a'_{ij}|\leq C|\rho|^\eta$. 

On the other hand, assuming that $f$ is independent of $y$ and $\theta$, we also combine the first term of this last 
equation (remember to reinsert the minus sign!) with the penultimate term $\Delta f$ in \eqref{eq:recentered-est}. 
These together yield
\[
\sum_{i=1}^\ell \big(-\pa_{\rho_i}^2+(k-2+a_i'')\rho_i^{-1}\pa_{\rho_i}\big)f+\sum_{i,j=1}^\ell a_{ij}''\pa^2_{\rho_i \rho_j} f,
\]
where $|a_i''|,|a_{ij}''|\leq C|\rho|^{\eta}$. 

We shall use, as before, the specific function
$f=f_\delta=(\prod_{i=1}^\ell\phi_i(\rho_i/\delta)^2)\rho^{2\sigma}$,
with $\phi_i(t)=(1+t^{-1})^{-\beta_i}$. Thus
\[
\begin{split} 
-\pa_{\rho_i}^2 f_i+\frac{k-2}{\rho_i} \pa_{\rho_i} f_i  = & f_i\rho_i^{-2}\left(-2\beta_i (2\beta_i+1)(1+t)^{-2}+ \right. \\
& \left. \left. 2\beta_i(k-4\sigma_i) (1+t)^{-1}+2\sigma_i(k-1-2\sigma_i)\right)\right|_{t=\rho_i/\delta}; 
\end{split}
\]
the remainder terms $a_i'' \rho_i^{-1}\pa_{\rho_i} f$, $a_{ij}''\pa^2_{\rho_i \rho_j} f$ are all bounded by 
$|\rho|^\eta$ times a linear combination of the three terms on the right. 
Taking $\beta_i=\frac{k-1}{2}-\sigma_i$, then for $\delta>0$, 
\[
\rho^{-2\sigma_0}f,\ \rho^{-2\sigma_0}\langle \rho \rangle\pa_{\rho_i} f,\ \mbox{and}\ \ \rho^{-2\sigma_0}\langle \rho \rangle^2\Delta f
\]
are all bounded, though not uniformly in $\delta$. Furthermore,
\begin{equation*}
\begin{split}
&\sum_{i=1}^\ell(-\pa_{\rho_i}^2 f+\frac{k-2}{\rho_i} \pa_{\rho_i} f) =\sum_{i=1}^\ell 2\beta_i f \rho_i^{-2} 
\left(-(2\beta_i+1)(1+t)^{-2} \right. \\  & \qquad \qquad \left. +(2\beta_i+1-2\sigma_i)   (1+t)^{-1}+2\sigma_i \right)|_{t=\rho_i/\delta} \\
&=\sum_{i=1}^\ell 2\beta_i f\rho_i^{-2}\frac{\rho_i/\delta}{1+\rho_i/\delta}\left(\frac{2\beta_i+1}{1+\rho_i/\delta}+2\sigma_i\right).
\end{split}
\end{equation*}
The right side of this equation is nonnegative if $\beta_i>0$ and $\sigma_i\geq 0$ for all $i$, and bounded from below by
\[
\sum_i 4\beta_i\sigma_i f\rho_i^{-2}\frac{\rho_i/\delta}{1+\rho_i/\delta}.
\]
Finally, using Cauchy-Schwarz, 
\begin{multline*}
\sum_{i, j} |\pa^2_{\rho_i \rho_j} f| +  \sum_i |\rho_i^{-1} \pa_{\rho_i} f |  \\ \leq  C f
\sum_{i, j} (\rho_i \rho_j)^{-1} \left((1+t)^{-2}+ (1+t)^{-1}+1\right)|_{t=\rho_i/\delta}\leq Cf \sum_{i} \rho_i^{-2}.
\end{multline*}

We have now proved that the last two terms of \eqref{eq:recentered-est} are bounded from below by
$- C|\rho|^{\eta}\sum (\rho_i\rho_j)^{-1}$.
Since we are assuming that $\sigma_i>0$ for every $i$ and $\eta\geq 1$, we can refine this since then
$|\rho|^\eta \leq |\rho| \leq \sum_j \rho_j$.  Indeed, using these two conditions, then for any specified $\ep_0>0$,Q
\[
\sum_j \rho_j \left((1+t)^{-2}+
(1+t)^{-1}+1\right)|_{t=\rho_i/\delta}\leq 3\sum_j\rho_j\leq \sum_j \ep_0 4\beta_j\sigma_j \frac{\rho_j/\delta}{1+\rho_j/\delta}
\]
when both $\delta$ and $|\rho|$ are sufficiently small.  Taking $\supp u$ sufficiently small so that both of these
last conditions hold, this shows that the last two terms of \eqref{eq:recentered-est} are 
actually non-negative.

Now, when $\delta>0$, \eqref{eq:recentered-est} gives 
\begin{equation}\begin{split}\label{eq:recentered-est-2}
&\langle f_\delta\Delta u,u\rangle+\langle\Delta (f_\delta u),u\rangle\\
&\qquad=2\|f_\delta^{1/2}(\rho^{-(k-1)}\nabla\rho^{k-1})u\|^2\\
&\qquad\qquad+\left\langle\left(\Delta f_\delta-2(\rho^{-(k-1)}\nabla\rho^{k-1})^*(f_\delta\rho^{-(k-1)}(\nabla\rho^{k-1}))\right)u,u\right\rangle
\end{split}\end{equation}
if $u\in\CI_c(\mathring \cM)$. Both sides are continuous bilinear forms on $X^{-\sigma_0}$, so since 
$\CI_c(\mathring \cM)$ is dense in $X^{-\sigma_0}$, the identity holds in this larger space. Since $\eta \geq 1$, 
the last term is nonnegative, hence 
$$
\|f_\delta^{1/2}(\rho^{-(k-1)}\nabla\rho^{k-1})u\|^2\leq
\|\rho^{\sigma_0}\langle \rho \rangle^{-1}u\|\,\|\rho^{2\sigma-\sigma_0}\langle \rho \rangle\Delta u\|.
$$
Letting $\delta\to 0$ shows $\rho^\sigma(\rho^{-(k-1)}\nabla\rho^{k-1}) u\in L^2$ and
$$
\|\rho^\sigma(\rho^{-(k-1)}\nabla\rho^{k-1}) u\|^2\leq \|\rho^{2\sigma-\sigma_0}\langle \rho \rangle\Delta u\| \|\rho^{\sigma_0}\langle 
\rho \rangle^{-1}u\|.
$$
This completes the proof. 
\end{proof}

\medskip

\begin{cor}\label{cor:indicial-bdy}
If $\sigma_0=\frac{k-1}{2}>0$ and $0 \leq \sigma < \sigma_0$, $\eta \geq 1$, 
and if $\rho^{\sigma_0}\langle \rho \rangle^{-1}u\in L^2$ and
$\rho^\sigma\langle \rho \rangle\Delta u\in L^2$ then $u\in X^\sigma$, so $\rho^\sigma\nabla u\in L^2$ and 
$\rho^\sigma\langle \rho \rangle^{-1}u\in L^2$.

In particular, if $\rho^{\sigma_0}\langle \rho \rangle^{-1}u\in L^2$ and $\langle \rho \rangle\Delta u\in L^2$
then $\nabla u\in L^2$.
\end{cor}
\begin{proof}
By Lemma~\ref{lemma:indicial-bdy}, if $\sigma'=\max(\sigma_0-1/\ell , \sigma_0-\frac{1}{2}(\sigma_0-\sigma))<\frac{k-1}{2}$,
$\sigma' > 0$,  
then in particular $2\sigma'-\sigma_0\geq \sigma$, and hence $\rho^{2\sigma'-\sigma_0}\langle \rho \rangle\Delta u\in L^2$, 
and we also have $\rho^{\sigma'}\rho^{-(k-1)}\nabla\rho^{k-1}u\in L^2$. This implies that 
\[
(\prod_{i\neq j}\rho_i^{\sigma'})\rho_j^{\sigma'}\rho_j^{-(k-1)}\pa_{\rho_j}\rho_j^{k-1}u\in L^2.
\]
Since $\rho_j^{1-k}$ barely fails to lie in $L^2$, while $u\in L^2$, we can use techniques of regular singular operator theory to
get that $\rho^{\sigma'}\pa_{\rho_j}u \in L^2$ for every $j$, and hence using the Hardy inequality, 
$\rho^{\sigma'}\langle \rho \rangle^{-1}u\in L^2$. This means that $u\in X^{-\sigma'}$.  Finally, by Lemma~\ref{lemma:indicial-gap} 
with $\sigma_0=\sigma'$, $\rho^\sigma\nabla u\in L^2$ and $\rho^\sigma\langle \rho \rangle^{-1}u\in L^2$, 
and this completes the proof. 
\end{proof}

We now show how to use the results above to establish that $\Delta$ is essentially self-adjoint.  
It is well-known that essential self-adjointness is equivalent to the density of $\Ran_{\CI_c}(\Delta\pm i)$ in $L^2$. 
If either of these spaces fail to be dense, then there exists an element $\psi\in L^2$ such that $\psi \perp \Ran_{\CI_c}(\Delta\pm i)$, 
i.e.\ $\langle (\Delta\pm i)\phi,\psi\rangle=0$ for all $\phi\in\CI_c$. This implies, in particular, that $\psi$ is a distributional solution 
of $(\Delta\pm i)\psi=0$, and hence $\psi\in\CI$, and in addition $\Delta\psi\in L^2$. 

If we can integrate by parts to justify the identity 
$$
0=\langle (\Delta\pm i)\psi,\psi\rangle=\|d\psi\|^2\pm i\|\psi\|^2,
$$
we could then conclude that $\psi=0$.  Thus it remains to prove the 
\begin{lem}
Suppose that $k\geq 3$, and let $g$ be a metric which satisfies the conditions \eqref{eq:metric-terms} and
\eqref{newconditions}, with $\eta \geq 1$ if $k=3$.  If $u\in L^2$ and $\Delta u\in L^2$, then $\nabla u\in L^2$ and
$\langle\Delta u,u\rangle=\langle \nabla u, \nabla u\rangle$.
\label{Tuesday}
\end{lem}
\begin{proof}
If $k>3$ then, by Lemma~\ref{lemma:indicial-gap} with $\sigma_0=1$ and $\sigma=0$, we
see that $\nabla u\in L^2$ and $\langle \rho \rangle^{-1}u\in L^2$. This proves the first claim.

Now suppose that $k=3$, which is the most relevant case. Then by Lemma~\ref{cor:indicial-bdy} 
with $\sigma_0=1$ and $\sigma=0$, we see that $\nabla u\in L^2$ and $\langle \rho \rangle^{-1}u\in L^2$.
This completes the proof of the first claim in all cases. 

Finally, if $v\in\CI_c(\mathring \cM)$, then 
$$
\langle\Delta v,v\rangle=\|\nabla v\|^2.
$$
Both sides are continuous bilinear forms on $X^0$, so by the density of 
$\CI_c(\mathring \cM_\gamma)$, this identity 
remains valid in $X^0$. In part one we proved that the assumption $u\in L^2$
and $\Delta u\in L^2$ implies that $\nabla u\in L^2$ and $\langle\rho\rangle^{-1}
u\in L^2$. Thus $u\in X^0$ and hence, the above equality holds for $u$.
\end{proof}

In summary, we proved the
\begin{thm}
$\Delta$ is essentially self-adjoint.
\end{thm}

To conclude this section, we note that the argument of Lemma~\ref{Tuesday} shows that the maximal
domain $\calD_{\max}(\Delta)$ is contained in $\langle \rho \rangle L^2 \cap H^1$, which is certainly
compactly contained in $L^2$. This proves Corollary~\ref{discspec}.

\section{The Weyl estimate for the eigenvalues}
In this final section we address the question of estimating the growth rate of the counting function
\[
N(\lambda) = \# \{j: \lambda_j \leq \lambda\}.
\]
The classical formula, valid for the Laplacian on a compact smooth manifold $(M,g)$, states that
\begin{equation}
N(\lambda) =\frac{\omega_n}{(2\pi)^n} \mathrm{Vol}_g(M) \lambda^{n/2} + 
o(\lambda^{n/2}), 
\label{Weyl}
\end{equation}
where $n=\dim M$ and $\omega_n$ is the volume of the unit ball in $\RR^n$. 
It is now well-understood, of course, that a good estimate of the error term, for example showing that it has
the form $\calO(\lambda^{n/2 - \epsilon})$ for some $\epsilon > 0$, depends on the dynamical properties of
the geodesic flow. The same is certainly true here, but we content ourselves with the most basic
result in this direction, which is the extension of this Weyl counting formula to the present setting. Our main result is the
\begin{thm}
Let $(M,g)$ be a compact stratified space with crossing cusp singularities of multi-order $k$, where
each $k_i \geq 3$. Then the spectrum of $\Delta_g$ is discrete, and the counting function $N(\lambda)$ for 
this spectrum satisfies the asymptotic formula \eqref{Weyl}.
\end{thm}
\begin{proof}
We follow the most classical method, known as Dirichlet-Neumann bracketing. We briefly recall the idea, 
but refer to \cite{Chavel} for details. Consider a partition of $\calM$ into the region 
\[
\calW_0 = \{p: \mbox{dist}_{g}(p, \cup D_j) \geq \delta\}
\]
and a finite number of regions $\calW_{\alpha}$, each of the form
\begin{multline*}
\{p = (r_1, \ldots, r_\ell, \theta_1, \ldots, \theta_\ell, y):  0 \leq r_i \leq \delta,\ i = 1, \ldots, \ell,\  \\ 
\theta_i \in S^1\ \forall\, i,\ \mbox{and}\ y \in \calV_\alpha \subset \RR^{6\gamma-6-2\ell}\}.
\end{multline*}

Next, define the Rayleigh quotient for $\Delta_g$ by $R(u) = D(u)/L(u)$, where $L(u) = \int |u|^2\, dV_g$ 
and $D(u) = \int |\nabla u|^2\, dV_g$ are the $L^2$ norm and the Dirichlet form.  We also consider the
restrictions of these forms to various subdomains $\calW$, and will denote these by $R^\calW(u)$, etc. 

For each $\alpha$, including $\alpha = 0$, restrict $R$ to functions $u$ which lie in $H^1(\calW_\alpha)$ or to 
$H^1_0(\calW_\alpha)$. The critical values of this restricted functional are the Neumann and Dirichlet
eigenvalues for $\Delta_g$ on $\calW_\alpha$, which we list in order, and with multiplicity, as
\[
\mu_{1,N}^\alpha \leq \mu_{2,N}^\alpha \leq \ldots \ \ \mbox{and}\qquad \mu_{1,D}^\alpha \leq \mu_{2,D}^\alpha \leq \ldots,
\]
respectively. The union of these lists of eigenvalues over all $\alpha$, reindexed into nondecreasing sequences, become
\[
\mu_{1,N} \leq \mu_{2,N} \leq \ldots \ \ \mbox{and}\qquad \mu_{1,D} \leq \mu_{2,D} \leq \ldots
\]
Since 
\[
\bigoplus_\alpha H^1_0(\calW_\alpha) \subset H^1(\calM) \subset \bigoplus_\alpha H^1(\calW_\alpha),
\]
the minimax characterization of eigenvalues shows that the eigenvalues $\{\lambda_j\}$ of $\Delta_g$ on
$\calM$, listed with multiplicity, satisfy 
\[
\mu_{j,N} \leq \lambda_j \leq \mu_{j,D}
\]
for all $j$. This implies, in turn, that if we denote the counting functions for the sequences $\{\mu_{j,N}\}$ and $\{\mu_{j,D}\}$
by $N_N(\lambda)$ and $N_D(\lambda)$, respectively, then for $\lambda \gg 0$, 
\[
N_D(\lambda) \leq N(\lambda) \leq N_N(\lambda). 
\]
Thus to prove the main assertion, it suffices to prove that both $N_N$ and $N_D$ satisfy this same asymptotic law. 

We take as given that
\[
N_{N/D}^{\calW_0}(\lambda) =\frac{\omega_n}{(2\pi)^n} \mathrm{Vol}_g(\calW_0) 
\lambda^{n/2} + o(\lambda^{n/2}).
\]
This is just the standard Weyl law for the region $\calW_0$ with either Neumann or Dirichlet eigenvalues. 
Since $\mathrm{Vol}_g(\calW_0) \to \mathrm{Vol}_g(M)$ as $\delta\to 0$,
 this is nearly the entire leading term.
On the other hand, since $N_D^{\calW_\alpha}(\lambda) \leq N_N^{\calW_\alpha}(\lambda)$ for $\alpha \neq 0$, 
it suffices to prove that for any $\epsilon > 0$ we can choose a sufficiently small $\delta$ so that for all such $\alpha$, 
\[
N_N^{\calW_\alpha}(\lambda) \leq \epsilon \lambda^{n/2} + \calO(\lambda^\beta)
\]
for some $\beta < n/2$. Thus we concentrate on this last estimate for any fixed $\alpha$. 

Since the Rayleigh quotient is changed by at most a bounded factor if we replace $g$ by any quasi-isometric
metric, we may replace $\calW_\alpha$ by the product $(0,\delta)^\ell \times Z$, $Z = (S^1)^\ell \times B$,
where $B$ is a piecewise smooth compact domain in $\RR^{n-2\ell}$, endowed with the warped product metric
\[
g_\ell =\sum_{i=1}^\ell (d\rho_i^2+\rho_i^{2k_i}d\theta^2)+dy^2,
\]
Using radial and angular variables $\rho_i$ and $\theta_i$, $i = 1, \ldots, \ell$, as before, as well as the
multi-index notation, so that for example $\rho^k = \prod_{i=1}^l\rho_i^{k_i}$, then
\[
L^{\calW_\alpha}(u) = \int_0^\delta \cdots \int_0^\delta \int_Z |u|^2 \rho^k \,d\rho\, d\theta\, dy,
\]
and 
\begin{multline*}
D^{\calW_\alpha}(u) \\
= \int_0^\delta \cdots \int_0^\delta \int_Z \left(\sum_{j=1}^\ell |\pa_{\rho_j}u|^2+\sum_{j=1}^\ell \rho_j^{-2k_j}|\pa_{\theta_j}u|^2
+|\pa_y u|^2\right) \rho^k \,d\rho\,d\theta\,dy.
\end{multline*}

Next, suppose that $\delta=2^{-m_0}$ for some $m_0$ which will be fixed later, and decompose the cube $(0,\delta)^\ell$ 
into a finite union of subregions
\[
(0, 2^{-m_0})^\ell = \bigsqcup_{\mu} I^\mu,\quad \mbox{where}\qquad I_\mu = I_{\mu_1} \times \ldots \times I_{\mu_\ell}.
\]
Here $\mu$ varies over all multi-indices $(\mu_1, \ldots, \mu_\ell)$ with $\mu_i \in \{m_0, \ldots, m+1\}$ 
for some $m$ to be chosen momentarily, and 
\[
I_j = (2^{-j-1}, 2^{-j}) \ \ \mbox{for}\ \  m_0 \leq j \leq m, \quad \mbox{and}\qquad I_{m+1} = (0, 2^{-m-1}). 
\]
In other words, this is just the dyadic decomposition of $(2^{-m-1}, 2^{-m_0})$ along with the `terminal' interval $(0, 2^{-m-1})$. 

Now fix $\lambda \gg 0$ and set $m \sim \frac12 \log_2 \lambda$.  We claim that if the multi-index $\mu$ has
$\mu_j = m+1$ for some $j$, then the number of Neumann eigenvalues of $\Delta_{g_\ell}$ on the region $I_\mu \times Z$ 
is bounded by the number of Neumann eigenvalues on the adjacent domain $I_{\mu'} \times Z$, where $\mu'_j = m$ 
and all other $\mu_i' = \mu_i$.  As we show below, it is possible to directly estimate the counting functions
on these non-terminal regions, and since we now show that the counting functions on the terminal regions are 
estimated in terms of these, we will have accounted for every block in this decomposition. 

To prove this we integrate in $\rho_j \in (0, 2^{-m})$ (for the same value of $j$). Writing $Y = \prod_{i \neq j} I_{\mu_i} \times Z$, 
then the second part of Lemma~\ref{lem-hardy} with $\beta = 0$, 
$\rho_0 = 2^{-m-1}$ and $\rho_0 + \epsilon = 2^{-m}$ gives that
\begin{multline*}
2^{2m+2} \int_0^{2^{-m-1}} |u|^2\, d\rho_j \leq \int_{0}^{2^{-m}} |\rho_j^{-1}u|^2 \, d\rho_j  \\ \leq 
C \int_0^{2^{-m}}  |\pa_{\rho_j}u|^2 \, d\rho_j + 2^{2m} \int_{2^{-m-1}}^{2^{-m}} |u|^2\, d\rho_j.
\end{multline*}
The constant in front of the last term on the right comes from the square of the derivative of the cutoff function 
used to reduce back to a function which vanishes near $\rho_j = \rho_0 + \epsilon$ so as to apply
the Hardy inequality for such functions. Integrating over the other factor $Y$, we obtain
\[
2^{2m+2} \int_{I_\mu \times Y} |u|^2 \leq C \int_{(I_\mu \cup I_{\mu'}) \times Y} 
|\nabla u|^2 + C 2^{2m} \int_{I_{\mu'} \times Y} |u|^2. 
\]

Now suppose that the restriction of $u$ to $I_{\mu'} \times Y$ is orthogonal to 
all the Neumann eigenfunctions 
with eigenvalues less than $C 2^{2m}$ on this region. We can then estimate the 
second term on the right by the Dirichlet integral
of $u$ on $I_{\mu'} \times Y$. So altogether, for such $u$,
\[
2^{2m+2} \int_{(I_\mu \cup I_{\mu'}) \times Y} |u|^2 \leq C \int_{(I_\mu 
\cup I_{\mu'}) \times Y} |\nabla u|^2,
\]
hence the Rayleigh quotient for all such $u$ satisfies $R(u) \geq C 2^{2m+2}$. 
This proves that
\[
N_N^{(I_\mu \cup I_{\mu'}) \times Y} (2^{2m+2}) \leq N_N^{ I_{\mu'} \times Y}(C 2^{2m+2}), 
\]
which proves the claim since $\lambda \sim 2^{2m}$. 

We have now reduced to estimating $N_N^{I_\mu \times Y}(\lambda)$ for any 
multi-index $\mu$ where $\mu_j \neq m+1$ 
for any $j$. To accomplish this, we consider the Rayleigh quotient on this 
region. Writing $ \mu \cdot k 
= \sum \mu_j k_j$, then 
\[
D(u) = \int (\sum_{j=1}^\ell(|\pa_{\rho_j} u|^2+2^{2k_j \mu_j}|\pa_{\theta_j} u|^2)
+|\pa_y u|^2) 2^{-\mu \cdot k} \,d\rho\,d\theta\,dy,
\]
and $L(u) = \int |u|^2 2^{-\mu \cdot k}\,d\rho\,d\theta\,dy$. Up to a 
$\mu$-independent constant factor, 
the Rayleigh quotient in this region is the same as the one for 
\[
\sum_{j=1}^\ell \pa^2_{\rho_j}+\sum_{j=1}^\ell 2^{2k_j \mu_j}\pa_{\theta_j}^2+\pa_y^2.
\]
Up to another constant factor, the eigenvalues of this problem are simply 
\[
\sum_{j=1}^\ell 2^{2\mu_j} \xi_j^2+\sum_{j=1}^\ell 2^{2 \mu_j k_j} \zeta_j^2+|\eta|^2,
\ \xi, \zeta \in\ZZ^\ell,\  \eta \in\ZZ^{n-2\ell}.
\]
The number of these which are no larger than $\lambda$ is estimated from above 
by the number of multi-indices 
$\xi$, $\zeta$, $\eta$ such that the individual summands themselves are less 
than $\lambda$.  Thus
\begin{multline*}
\#\{ \xi_j: 2^{2\mu_j}\xi_j^2 \leq \lambda\} \leq C \sqrt{\lambda} 2^{-\mu_j},\\ 
\#\{ \zeta_j: 2^{2k_j \mu_j}\zeta_j^2 \leq \lambda\} \leq C \sqrt{\lambda} 2^{-\mu_j k_j}, \\
\mbox{and} \qquad \# \{\eta_j: \eta_j^2 \leq \lambda\} \leq C \sqrt{\lambda}. 
\end{multline*}
Thus, summing over all $\mu$ (with no $\mu_j = m$), and recalling that $\lambda$ is large, 
the number of eigenvalues of these model problems on the various regions $I_\mu \times Z$ 
which are less than $\lambda$ is bounded by
\[
C\left( \sum_\mu 2^{-(1+k_j)\mu_j} \right) \lambda^{n/2}.
\]
The coefficient breaks into the product
\[
\prod_{j=1}^\ell \left(\sum_{\mu_j = m_0}^{m}2^{-(1+k_j) \mu_j} \right)  \leq \prod_{j=1}^{\ell} 2^{-(1+k_j) m_0} = 2^{-(\ell + |k|)m_0} 
= \delta^{\ell + |k|} 
\]
Altogether, we have proved that the counting function on each $\calW_\alpha$, $\alpha \neq 0$, is bounded
by $C\delta^\beta \lambda^{n/2}$ where $\beta = \ell + |k|$. 

We have now shown that
\[
N(\lambda) =\frac{\omega_n}{(2\pi)^n} \left(\mathrm{Vol}_g(M) + \calO(\delta^\beta)\right)\lambda^{n/2} + o(\lambda^{n/2}). 
\]
If we choose $\delta = (\log_2 \lambda)^{-1/\beta}$, or equivalently, $m_0 = (1/\beta) \log_2 \lambda$, then
since this $m_0$ is still far less than $m = \frac12 \log_2 \lambda$, we conclude that 
\[
N(\lambda) =\frac{\omega_n}{(2\pi)^n} \mathrm{Vol}_g(M)\lambda^{n/2} + o(\lambda^{n/2}), 
\]
which finishes the proof.
\end{proof}

\end{document}